\documentclass[11pt,letterpaper]{article}
\usepackage[T1]{fontenc}
\usepackage{lmodern}
\usepackage[margin=0.9in]{geometry}
\usepackage{amsmath,amssymb,amsthm,mathtools}
\usepackage{enumitem,microtype}
\usepackage[hidelinks,hyperfootnotes=false]{hyperref}
\usepackage{needspace}
\newtheoremstyle{paper}{7pt}{5pt}{\normalfont}{} {\bfseries}{.}{.5em}{}
\theoremstyle{paper}
\newtheorem{theorem}{Theorem}[section]
\newtheorem{proposition}[theorem]{Proposition}
\newtheorem{lemma}[theorem]{Lemma}
\newtheorem{corollary}[theorem]{Corollary}
\newtheorem{definition}[theorem]{Definition}

\newtheorem{remark}[theorem]{Remark}
\newtheorem*{theoremA}{Theorem A}
\newtheorem*{theoremB}{Theorem B}
\newtheorem*{theoremC}{Theorem C}
\newtheorem*{theoremD}{Theorem D}
\newtheorem*{question}{Question}
\newcommand{\OO}{\mathcal O}
\newcommand{\FF}{\mathcal F}
\newcommand{\GG}{\mathcal G}
\newcommand{\II}{\mathcal I}
\newcommand{\PP}{\mathbb P}
\newcommand{\CC}{\mathbb C}
\newcommand{\QQ}{\mathbb Q}
\newcommand{\ZZ}{\mathbb Z}
\newcommand{\simQ}{\sim_{\QQ}}
\DeclareMathOperator{\Pic}{Pic}
\DeclareMathOperator{\pr}{pr}
\DeclareMathOperator{\Bl}{Bl}
\DeclareMathOperator{\rk}{rk}
\DeclareMathOperator{\codim}{codim}
\DeclareMathOperator{\Ext}{Ext}

\DeclareMathOperator{\length}{length}
\title{\bfseries Log Canonical Counterexamples to the Chen--Jiang Decomposition}
\author{Houari Benammar Ammar}
\date{}
\begin{document}
\maketitle
\begingroup
\renewcommand{\thefootnote}{}
\endgroup
\begin{abstract}
We construct a projective log-smooth surface pair $(X,\Delta)$ and a morphism $f:X\to E$ to an elliptic curve such that $f_*\OO_X(K_X+\Delta)$ is Atiyah's indecomposable rank-two bundle $F_2$ of degree zero. This gives a negative answer to Question 3.4 of Jiang \cite{Jiang}. We also construct a log canonical pair $(X,\Delta)$ and a morphism $f$ to an abelian surface such that $f_*\OO_X(m(K_X+\Delta))$ does not admit a Chen--Jiang decomposition for any integer $m\ge2$. Products yield counterexamples in every higher direct image degree and over abelian bases of arbitrary positive dimension.
\end{abstract}

\section{Introduction}
Throughout, we work over the field of complex numbers $\CC$. Let $A$ be an abelian variety and let $\widehat A=\Pic^0(A)$. For a coherent sheaf $\FF$ on $A$, set
\[
V^i(A,\FF)=\{\alpha\in\widehat A\mid H^i(A,\FF\otimes\alpha)\ne0\}.
\]
A coherent sheaf $\FF$ is called a GV-sheaf \cite[\S2]{PP} if
\[
\codim_{\widehat A}V^i(A,\FF)\ge i\qquad(i>0),
\]
and it is called $M$-regular \cite[\S2]{PP} if
\[
\codim_{\widehat A}V^i(A,\FF)>i\qquad(i>0).
\]
We also say that $\FF$ satisfies $\mathrm{IT}^0$ if $H^i(A,\FF\otimes\alpha)=0$ for every $i>0$ and every $\alpha\in\Pic^0(A)$.

\begin{definition}[Chen--Jiang decomposition {\cite[Theorem 1.1]{CJ}}]
A coherent sheaf $\FF$ on an abelian variety $A$ is said to admit a Chen--Jiang decomposition if
\[
\FF\simeq\bigoplus_{j=1}^N\bigl(\alpha_j\otimes p_j^*\FF_j\bigr),
\]
where each $p_j:A\to A_j$ is a surjective morphism of abelian varieties with connected fibers, each $\alpha_j\in\Pic^0(A)$ is torsion, and each $\FF_j$ is a nonzero $M$-regular coherent sheaf on $A_j$.
\end{definition}

The decomposition first appeared in the work of Chen and Jiang in the setting of varieties of maximal Albanese dimension \cite[Theorem 1.1]{CJ}. Pareschi, Popa and Schnell established a general canonical decomposition theorem, including higher direct images, using Hodge modules on complex tori \cite{PPS}. Popa and Schnell proved generic vanishing and positivity results for pluricanonical direct images \cite{PS}, while Lombardi, Popa and Schnell obtained the Chen--Jiang decomposition for pluricanonical direct images of smooth projective varieties \cite{LPS}.

For klt pairs, Jiang and Meng separately established the following decomposition results.

\begin{theorem}[Jiang--Meng {\cite[Theorem 1.3]{Jiang}; \cite[Theorem 1.3 and Proposition 3.1]{Meng}}]
Let $(X,\Delta)$ be a projective klt pair, let $f:X\to A$ be a morphism to an abelian variety, and let $D$ be a Cartier divisor on $X$.
\begin{enumerate}[label=(\roman*),leftmargin=2em]
\item If $D\simQ K_X+\Delta$, then, for every $i\ge0$, there is a finite index set $J_i$ and a decomposition
\[
R^if_*\OO_X(D)\simeq\bigoplus_{j\in J_i}\alpha_{i,j}\otimes p_{i,j}^*\FF_{i,j}.
\]
Here, for each $j\in J_i$, the map $p_{i,j}:A\to A_{i,j}$ is a surjective morphism of abelian varieties with connected fibers, $\FF_{i,j}$ is a nonzero $M$-regular coherent sheaf on $A_{i,j}$, and $\alpha_{i,j}\in\Pic^0(A)$ is a torsion line bundle. In particular, this holds for $f_*\OO_X(D)$ by taking $i=0$.
\item If $D\simQ m(K_X+\Delta)$ for a rational number $m\ge2$, then there is a finite index set $J$ and a decomposition
\[
f_*\OO_X(D)\simeq\bigoplus_{j\in J}\alpha_j\otimes p_j^*\FF_j.
\]
Here, for each $j\in J$, the map $p_j:A\to A_j$ is a surjective morphism of abelian varieties with connected fibers, $\FF_j$ is a nonzero $M$-regular coherent sheaf on $A_j$, and $\alpha_j\in\Pic^0(A)$ is a torsion line bundle.
\end{enumerate}
\end{theorem}

The higher-direct-image assertion does not extend in general to pluricanonical bundles. Shibata \cite[Example 4.5]{Shibata} constructed a smooth projective variety $X$ and a morphism $f:X\to A$ to an abelian variety for which $R^if_*\omega_X^{\otimes m}$ is not GV for some integers $i>0$ and $m>1$. Since a sheaf admitting a Chen--Jiang decomposition is GV, these higher direct images do not admit such a decomposition.

This leads to the following question about log canonical pairs.
\begin{question}[Jiang {\cite[Question 3.4]{Jiang}}]
Let $(X,\Delta)$ be a projective log canonical pair, let $f:X\to A$ be a morphism to an abelian variety, and let $D$ be a Cartier divisor such that
\[
D\simQ m(K_X+\Delta)
\]
for an integer $m\ge1$. Does $f_*\OO_X(D)$ admit a Chen--Jiang decomposition?
\end{question}

Let $F_2$ denote Atiyah's bundle defined by the unique non-split extension
\[
0\longrightarrow\OO_E\longrightarrow F_2\longrightarrow\OO_E\longrightarrow0
\]
on an elliptic curve $E$. The following theorem answers Jiang's question negatively for $m=1$.

\begin{theoremA}[= Theorem \ref{thm:lc}; main result]
Let $E$ be an elliptic curve. There exist a smooth projective surface $X$, a reduced simple-normal-crossings divisor $\Delta$ on $X$, and a morphism $f:X\to E$ such that
\[
f_*\OO_X(K_X+\Delta)\simeq F_2.
\]
Consequently $f_*\OO_X(K_X+\Delta)$ does not admit a Chen--Jiang decomposition.
\end{theoremA}

The failure also occurs for every integer $m\ge2$. In this case the counterexample is a rank-one torsion-free sheaf on an abelian surface.

\begin{theoremB}[= Theorem \ref{thm:pluri}; second main result]
Let $E$ be an elliptic curve, set $A_0=E\times E$, and let $p=(0,0)$. Put
\[
L=\OO_{A_0}(\{0\}\times E+E\times\{0\}).
\]
There exist a smooth projective surface $X$, a reduced SNC divisor $\Delta$ on $X$, and a birational morphism $\mu:X\to A_0$ such that, for every integer $m\ge2$,
\[
\GG_m:=\mu_*\OO_X\bigl(m(K_X+\Delta)\bigr)\simeq L^m\otimes\II_p^m
\]
is GV but does not admit a Chen--Jiang decomposition.
\end{theoremB}

\begin{theoremC}[= Corollary \ref{cor:all}; higher direct images and higher-dimensional bases]
For every pair of integers $d\ge1$ and $j\ge0$, there exist a projective log-smooth pair $(W,\Delta_W)$ with reduced boundary, an abelian variety $B_d$ of dimension $d$, and a morphism $\varphi:W\to B_d$ such that
\[
\dim W=d+j+1,\qquad R^j\varphi_*\OO_W(K_W+\Delta_W)\simeq p^*F_2,
\]
where $p:B_d\to E$ is a quotient onto an elliptic curve. In particular, this higher direct image does not admit a Chen--Jiang decomposition.
\end{theoremC}

In Theorem D, we follow the notation of Theorem B.

\begin{theoremD}[= Theorem \ref{thm:pluriproduct}; higher-dimensional examples for $m\ge2$]
Let $\mu:X'\to A_0=E\times E$ and $(X',\Delta')$ be the construction of Theorem B. For every $d\ge2$ and $r\ge0$, choose abelian varieties $C$ and $T$ with $\dim C=d-2$ and $\dim T=r$, and set
\[
B_d=A_0\times C,\qquad W=X'\times C\times T,\qquad\Delta_W=\pr_{X'}^*\Delta'.
\]
Let $p:B_d\to A_0$ be the projection and define
\[
\varphi:W\longrightarrow B_d,\qquad\varphi(x,c,t)=(\mu(x),c).
\]
Then $(W,\Delta_W)$ is log smooth with reduced boundary, $\dim W=d+r$, and $\varphi$ is surjective with connected fibers. Furthermore, for every integer $m\ge2$,
\[
R^j\varphi_*\OO_W\bigl(m(K_W+\Delta_W)\bigr)
\simeq(p^*\GG_m)^{\oplus\binom rj}\qquad(0\le j\le r).
\]
None of these sheaves admits a Chen--Jiang decomposition, and the higher direct images vanish for $j>r$.
\end{theoremD}
\begin{remark}
Theorem D can be viewed as complementary to Shibata's work since the higher direct images constructed here are GV-sheaves, but they do not admit a Chen--Jiang decomposition.
\end{remark}
For Theorem A, we blow up the intersection of the sections $\{0\}\times E$ and $\Delta_E$ in $E\times E$ and then compose the blow-up morphism with the second projection to $E$. The resulting direct image is a non-split extension of $\OO_E$ by itself. For Theorem B, we blow up the intersection of the two coordinate elliptic curves. The direct image on $A_0$ is $L^m\otimes\II_p^m$, and a restriction exact sequence shows that it is not $M$-regular. Its rank and failure of local freeness exclude a Chen--Jiang decomposition.

\noindent\textbf{Acknowledgment.} I would like to thank NCTS for its hospitality during my visit in March 2026, and Jungkai Chen for the invitation. I also thank Xi Chen and Steven Lu for general mathematical discussions. After the first version was posted on Arxiv, Fanjun Meng brought to my attention his paper \cite{MengAdjoint} on polarized pairs and informed me that he recently obtained an independent proof of Theorem A, which he will plan to include in a forthcoming paper. I would like to thank him for informing me and for his interest.

\noindent\textbf{Declaration on the use of artificial intelligence.} I benefited from mathematical discussions with ChatGPT 5.6 Sol. I started thinking about this question in 2024. The ideas were developed by myself, and I am responsible for the correctness of the proofs and the content of the manuscript. ChatGPT 5.6 helped me with formatting and LaTeX.

\section{The motivational example}
As motivation, We first give a ruled-surface construction realizing $F_2$ as an adjoint direct image with a boundary coefficient greater than one (see Remark \ref{remark2.2} below for the history of this construction).
\begin{proposition}\label{ex:ruled}
Let $E$ be an elliptic curve and let
\[
0\longrightarrow\OO_E\longrightarrow F_2\longrightarrow\OO_E\longrightarrow0
\]
be the non-split extension. Let $\pi:Z=\PP_E(F_2)\to E$ be the projective bundle, and let $S\subset Z$ be the section defined by the quotient $F_2\twoheadrightarrow\OO_E$. Then
\[
\OO_Z(S)\simeq\OO_Z(1),\qquad \pi_*\OO_Z(S)\simeq F_2,\qquad K_Z\sim-2S.
\]
For every rational $m>0$, the divisor $B=(2+1/m)S$ satisfies $S\simQ m(K_Z+B)$. Furthermore, if $B'\ge0$ is a rational divisor such that $S\simQ m(K_Z+B')$, then $B'=(2+1/m)S$. In particular, $(Z,B')$ is not log canonical.
\end{proposition}
\begin{proof}
The quotient $F_2\twoheadrightarrow\OO_E$ defines a section $S$ of $\pi$. Since the kernel is trivial, the divisor associated with this section is the tautological divisor, and hence $\OO_Z(S)\simeq\OO_Z(1)$. The projective bundle and canonical bundle formulas give
\[
\pi_*\OO_Z(S)\simeq F_2,\qquad K_Z\sim-2S+\pi^*(K_E+\det F_2)\sim-2S.
\]
Thus $m(K_Z+(2+1/m)S)\simQ S$.

Now suppose $B'\ge0$ and $S\simQ m(K_Z+B')$. Then
\[
B'\simQ\left(2+\frac1m\right)S.
\]
Choose an integer $q>0$ sufficiently divisible that $qB'$ is integral and
\[
n=q\left(2+\frac1m\right)\in\ZZ,\qquad qB'\sim nS.
\]
In characteristic zero, Atiyah's classification \cite{Atiyah} gives
\[
\operatorname{Sym}^nF_2\simeq F_{n+1},\qquad h^0(E,F_{n+1})=1.
\]
Consequently
\[
h^0(Z,\OO_Z(nS))=h^0(E,\operatorname{Sym}^nF_2)=1.
\]
Since $nS$ is effective, the complete linear system $|nS|$ has exactly one element. Since $nS$ itself belongs to $|nS|$, that unique divisor is $nS$. Hence $qB'=nS$, so $B'=(2+1/m)S$. Its coefficient along $S$ is greater than one.
\end{proof}

\begin{remark}\label{remark2.2}
The author was inspired by Hu \cite[Example 4.5]{Hu} to give Proposition \ref{ex:ruled}. The author recently learned that Meng gives a related example for polarized pairs \cite[Example 4.3]{MengAdjoint}, showing that the GV conclusion of \cite[Theorem 1.4]{MengAdjoint} does not imply a Chen--Jiang decomposition.
\end{remark}

\section{The log canonical case}
\begin{proposition}\label{prop:subsheaf}
Let $(X,\Delta)$ be a projective log canonical pair such that $X$ is klt and $\Delta$ is a reduced Cartier divisor. Let $f:X\to A$ be a morphism to an abelian variety, and let $m\ge2$ be an integer such that $D_m=m(K_X+\Delta)$ is Cartier. Then the natural subsheaf
\[
\FF_m:=f_*\OO_X(D_m-\Delta)\ \subset\ f_*\OO_X(D_m)
\]
admits a Chen--Jiang decomposition whenever it is nonzero.
\end{proposition}
\begin{proof}
We have
\[
m(K_X+\Delta)=m\left(K_X+\frac{m-1}{m}\Delta\right)+\Delta,
\]
and hence
\[
D_m-\Delta=m\left(K_X+\frac{m-1}{m}\Delta\right).
\]
Since $\Delta$ is an effective Cartier divisor, tensoring its ideal sequence by $\OO_X(D_m)$ gives
\begin{equation}\label{eq:boundarysequence}
0\longrightarrow\OO_X(D_m-\Delta)\longrightarrow\OO_X(D_m)
\longrightarrow\OO_\Delta(D_m)\longrightarrow0.
\end{equation}
Taking direct images yields the exact sequence
\[
0\longrightarrow\FF_m\longrightarrow f_*\OO_X(D_m)
\longrightarrow(f|_\Delta)_*\OO_\Delta(D_m)
\longrightarrow R^1f_*\OO_X(D_m-\Delta).
\]
Thus $\FF_m$ is a subsheaf of $f_*\OO_X(D_m)$.

Put $\Delta_m=(m-1)\Delta/m$. For every prime divisor $T$ over $X$, its log discrepancy satisfies
\[
a(T,X,\Delta_m)=\frac1m a(T,X,0)+\frac{m-1}{m}a(T,X,\Delta)>0,
\]
because $X$ is klt and $(X,\Delta)$ is lc. Therefore $(X,\Delta_m)$ is klt. The divisor $D_m-\Delta$ is Cartier and equals $m(K_X+\Delta_m)$, so $\FF_m$ admits the Chen--Jiang decomposition.
\end{proof}

\begin{remark}\label{rem:ammarsubsheaf}
The existence of a Chen--Jiang subsheaf in an lc direct image was considered in the author's unrevised preprint \cite[Proposition 1.4]{Ammar}, for sufficiently large and divisible multiples under a nonnegative fiberwise Kodaira-dimension hypothesis. There, the author predicts that the best result we could obtain is the existence of a higher-rank subsheaf that admits a Chen--Jiang decomposition.
\end{remark}

\begin{theorem}[Log canonical counterexample]\label{thm:lc}
Let $E$ be an elliptic curve. There exist a smooth projective surface $X$, a reduced simple-normal-crossings divisor $\Delta$ on $X$, and a morphism $f:X\to E$ such that
\[
f_*\OO_X(K_X+\Delta)\simeq F_2.
\]
In particular, $f_*\OO_X(K_X+\Delta)$ does not admit a Chen--Jiang decomposition.
\end{theorem}
\begin{proof}
Fix an elliptic curve $E$ with origin $0$ and set
\[
Y=E\times E,\qquad g=\pr_2:Y\longrightarrow E.
\]
The curves $C_0=\{0\}\times E$ and $C_\Delta=\Delta_E$ meet transversely at the unique point $p=(0,0)$. Let
\[
\mu:X=\Bl_p(Y)\longrightarrow Y
\]
be the blow-up, with exceptional curve $R$. Denote the strict transforms of $C_0$ and $C_\Delta$ by $D_0$ and $D_\Delta$, and put
\[
\Delta=D_0+D_\Delta,\qquad f=g\circ\mu:X\longrightarrow E.
\]
The curves $D_0$ and $D_\Delta$ are smooth and disjoint. Hence $(X,\Delta)$ is log smooth and therefore log canonical.

Set $L=\OO_Y(C_0+C_\Delta)$. Since $K_Y\sim0$,
\[
K_X=\mu^*K_Y+R\sim R,\qquad \mu^*C_0=D_0+R,\qquad \mu^*C_\Delta=D_\Delta+R.
\]
It follows that
\[
K_X+\Delta\sim\mu^*(C_0+C_\Delta)-R.
\]
Using $\mu_*\OO_X(-R)=\II_p$ and the projection formula, we obtain
\[
\mu_*\OO_X(K_X+\Delta)\simeq L\otimes\II_p.
\]
Thus, for $F=f_*\OO_X(K_X+\Delta)$,
\begin{equation}\label{eq:F}
F\simeq g_*(L\otimes\II_p).
\end{equation}
The inclusion $p\subset C_\Delta\subset Y$ gives $\II_{C_\Delta}\subset\II_p$, with quotient the ideal of $p$ in $C_\Delta$. If $i:C_\Delta\hookrightarrow Y$ denotes the inclusion, the resulting sequence is
\begin{equation}\label{eq:ideal}
0\longrightarrow\OO_Y(-C_\Delta)\longrightarrow\II_p\longrightarrow i_*\OO_{C_\Delta}(-p)\longrightarrow0.
\end{equation}
In completed local coordinates $x,y$ at $p$ with $\II_{C_\Delta}=(y)$ and $\II_p=(x,y)$, this quotient is
\[
(x,y)/(y)\simeq(x)\subset\CC[[x]],
\]
the ideal of the origin on $C_\Delta$. Tensoring \eqref{eq:ideal} by $L$ yields
\begin{equation}\label{eq:restriction}
0\longrightarrow\OO_Y(C_0)\longrightarrow L\otimes\II_p
\longrightarrow i_*\bigl(L|_{C_\Delta}\otimes\OO_{C_\Delta}(-p)\bigr)\longrightarrow0.
\end{equation}
Since $C_0$ meets $C_\Delta$ transversely at $p$, we have
\[
\OO_Y(C_0)|_{C_\Delta}\simeq\OO_{C_\Delta}(p).
\]
Along $C_\Delta\simeq E$, the differential of the diagonal embedding is $v\mapsto(v,v)$, so
\[
N_{C_\Delta/Y}\simeq(T_E\oplus T_E)/T_E^{\mathrm{diag}}\simeq T_E,
\]
where the last isomorphism is induced by $(u,v)\mapsto v-u$. A nonzero translation-invariant vector field trivializes $T_E$. Therefore
\[
\OO_Y(C_\Delta)|_{C_\Delta}\simeq N_{C_\Delta/Y}\simeq\OO_{C_\Delta},
\qquad L|_{C_\Delta}\otimes\OO_{C_\Delta}(-p)\simeq\OO_{C_\Delta}.
\]
Moreover, $\OO_Y(C_0)=\pr_1^*\OO_E(0)$ and $g=\pr_2$. The product formula gives
\[
Rg_*\pr_1^*\OO_E(0)\simeq R\Gamma(E,\OO_E(0))\otimes\OO_E.
\]
The degree-one line bundle $\OO_E(0)$ satisfies $h^0(E,\OO_E(0))=1$ and $h^1(E,\OO_E(0))=0$. Hence
\begin{equation}\label{eq:product}
g_*\OO_Y(C_0)\simeq\OO_E,\qquad R^1g_*\OO_Y(C_0)=0.
\end{equation}
Since $g|_{C_\Delta}$ is an isomorphism, pushing forward \eqref{eq:restriction} gives
\begin{equation}\label{eq:extension}
0\longrightarrow\OO_E\longrightarrow F\longrightarrow\OO_E\longrightarrow0.
\end{equation}
To determine this extension, consider the automorphism
\[
T:E\times E\longrightarrow E\times E,\qquad T(x,y)=(x,y-x).
\]
It sends $C_0$ to $\{0\}\times E$ and $C_\Delta$ to $E\times\{0\}$. Thus
\[
L\simeq T^*\bigl(\pr_1^*\OO_E(0)\otimes\pr_2^*\OO_E(0)\bigr),
\qquad h^0(Y,L)=h^0(E,\OO_E(0))^2=1.
\]
The unique section of $L$, up to scalar, has divisor $C_0+C_\Delta$ and vanishes at $p$. Consequently
\[
h^0(E,F)=h^0(Y,L\otimes\II_p)=1.
\]
If \eqref{eq:extension} were split, $F\simeq\OO_E^{\oplus2}$ would have two independent global sections. Since
\[
\Ext^1_E(\OO_E,\OO_E)\simeq H^1(E,\OO_E)\simeq\CC,
\]
we conclude that $F\simeq F_2$.

Finally, a quotient of $E$ with connected fibers is either $E$ itself or a point. A Chen--Jiang summand pulled back from $E$ is a torsion twist of a nonzero $M$-regular bundle and has positive degree. Since $\deg F_2=0$, all summands would have to come from a point. Then $F_2$ would be a direct sum of torsion line bundles, contrary to its indecomposability.
\end{proof}

\begin{remark}[Compatibility with generic vanishing]\label{rem:F2loci}
For the bundle $F_2$ above,
\[
V^0(E,F_2)=V^1(E,F_2)=\{\OO_E\}.
\]
Indeed, for every nontrivial $\alpha\in\Pic^0(E)$ one has $H^i(E,F_2\otimes\alpha)=0$ for $i=0,1$, while $h^0(E,F_2)=h^1(E,F_2)=1$. Thus $F_2$ is GV but not $M$-regular. The example is compatible with the generic-vanishing and torsion-translate properties of logarithmic pluricanonical direct images.
\end{remark}

\begin{corollary}\label{cor:torsion}
Let $E$ be an elliptic curve. There exist a smooth projective surface $X$, a reduced simple-normal-crossings divisor $\Delta$ on $X$, and a morphism $f:X\to E$ such that, for every torsion line bundle $\alpha\in\Pic^0(E)$, there is a Cartier divisor $D_\alpha$ on $X$ satisfying
\[
D_\alpha\simQ K_X+\Delta,\qquad f_*\OO_X(D_\alpha)\simeq F_2\otimes\alpha.
\]
In particular, $f_*\OO_X(D_\alpha)$ does not admit a Chen--Jiang decomposition.
\end{corollary}
\begin{proof}
Let $(X,\Delta)$ and $f:X\to E$ be the pair and the morphism constructed in Theorem \ref{thm:lc}. Thus
\[
f_*\OO_X(K_X+\Delta)\simeq F_2.
\]
Fix a torsion line bundle $\alpha\in\Pic^0(E)$. Since $E$ is smooth, we may choose a Cartier divisor $T_\alpha$ on $E$ such that
\[
\OO_E(T_\alpha)\simeq\alpha.
\]
Define a Cartier divisor on $X$ by
\[
D_\alpha:=K_X+\Delta+f^*T_\alpha.
\]
Here $K_X+\Delta$ is Cartier because $X$ is smooth and $\Delta$ is integral. No effectivity condition is required on $T_\alpha$ or $D_\alpha$.

We first check the asserted $\QQ$-linear equivalence. If $N>0$ is the order of $\alpha$, then $\alpha^{\otimes N}\simeq\OO_E$, so $NT_\alpha\sim0$. Pulling back this linear equivalence gives
\[
N\bigl(D_\alpha-(K_X+\Delta)\bigr)=f^*(NT_\alpha)\sim0.
\]
By the definition of $\QQ$-linear equivalence, this proves $D_\alpha\simQ K_X+\Delta$.

The associated line bundle satisfies
\[
\OO_X(D_\alpha)\simeq\OO_X(K_X+\Delta)\otimes f^*\alpha.
\]
Since $\alpha$ is locally free, the projection formula applies and yields
\begin{align*}
f_*\OO_X(D_\alpha)
&\simeq f_*\bigl(\OO_X(K_X+\Delta)\otimes f^*\alpha\bigr)\\
&\simeq f_*\OO_X(K_X+\Delta)\otimes\alpha\\
&\simeq F_2\otimes\alpha.
\end{align*}
We next show that this direct image has no Chen--Jiang decomposition. Suppose, to the contrary, that
\[
F_2\otimes\alpha\simeq\bigoplus_{j\in J}\beta_j\otimes p_j^*G_j,
\]
where $p_j:E\to A_j$ is a quotient morphism with connected fibers, $G_j$ is a nonzero $M$-regular coherent sheaf on $A_j$, and $\beta_j\in\Pic^0(E)$ is torsion. Tensoring this isomorphism by $\alpha^{-1}$ gives
\[
F_2\simeq\bigoplus_{j\in J}(\beta_j\otimes\alpha^{-1})\otimes p_j^*G_j.
\]
Each $\beta_j\otimes\alpha^{-1}$ is torsion, since both factors have finite order. The quotient maps $p_j$ and the $M$-regular sheaves $G_j$ are unchanged. This is therefore a Chen--Jiang decomposition of $F_2$, contradicting Theorem \ref{thm:lc}.
\end{proof}

\begin{remark}
Varying $\alpha$ gives infinitely many pairwise nonisomorphic direct images, since Remark \ref{rem:F2loci} yields $V^0(E,F_2\otimes\alpha)=\{\alpha^{-1}\}$. Consequently, the line bundles $\OO_X(D_\alpha)$ are also pairwise nonisomorphic.
\end{remark}

\begin{theorem}[Counterexamples for every $m\ge2$]\label{thm:pluri}
Let $E$ be an elliptic curve with origin $0$, and put
\[
A_0=E\times E,\qquad C_1=\{0\}\times E,\qquad C_2=E\times\{0\}.
\]
Let $p=(0,0)$, let $\mu:X=\Bl_p(A_0)\to A_0$ be the blow-up, and let $\Delta$ be the sum of the strict transforms of $C_1$ and $C_2$. Put $L=\OO_{A_0}(C_1+C_2)$. Then $(X,\Delta)$ is log smooth with reduced boundary, and for every integer $m\ge2$,
\[
\GG_m:=\mu_*\OO_X\bigl(m(K_X+\Delta)\bigr)\simeq L^m\otimes\II_p^m
\]
is GV but does not admit a Chen--Jiang decomposition.
\end{theorem}
\begin{proof}
Let $R$ be the exceptional curve and $D_1,D_2$ the strict transforms of $C_1,C_2$. The curves $D_1,D_2$ are smooth and disjoint, so $\Delta=D_1+D_2$ is reduced SNC. Since $K_{A_0}\sim0$ and $\mu^*C_i=D_i+R$, we have
\[
K_X+\Delta\sim\mu^*(C_1+C_2)-R.
\]
Therefore
\[
m(K_X+\Delta)\sim\mu^*m(C_1+C_2)-mR.
\]
Using $\mu_*\OO_X(-mR)=\II_p^m$ and the projection formula gives
\[
\GG_m\simeq L^m\otimes\II_p^m.
\]
This sheaf is GV by generic vanishing for lc pairs \cite[Variant 5.5]{PS}.

Write $P=\OO_E(0)$, so $L=\pr_1^*P\otimes\pr_2^*P$. Let $i:C_1\hookrightarrow A_0$ be the inclusion. Restriction to $C_1$ gives the exact sequence
\begin{equation}\label{eq:residual}
0\longrightarrow\OO_{A_0}(-C_1)\otimes\II_p^{m-1}
\longrightarrow\II_p^m\longrightarrow i_*\OO_{C_1}(-mp)\longrightarrow0.
\end{equation}
Indeed, in local parameters $x,y$ at $p$ with $C_1=(x=0)$, the image of $(x,y)^m$ in $\CC[[y]]$ is $(y^m)$ and its kernel is $x(x,y)^{m-1}$.

Fix $\beta\in\Pic^0(E)$ and put $\alpha=\pr_1^*\beta$. Since $L|_{C_1}\simeq\OO_{C_1}(p)$ and $\alpha|_{C_1}\simeq\OO_{C_1}$, tensoring \eqref{eq:residual} by $L^m\otimes\alpha$ gives
\begin{equation}\label{eq:nonvanishingsequence}
0\longrightarrow\mathcal K_{m,\beta}\longrightarrow\GG_m\otimes\alpha
\longrightarrow i_*\OO_{C_1}\longrightarrow0,
\end{equation}
where
\[
\mathcal K_{m,\beta}=
\bigl(\pr_1^*(P^{m-1}\otimes\beta)\otimes\pr_2^*P^m\bigr)\otimes\II_p^{m-1}.
\]
Set $N_{m,\beta}=\pr_1^*(P^{m-1}\otimes\beta)\otimes\pr_2^*P^m$. From
\[
0\longrightarrow\mathcal K_{m,\beta}\longrightarrow N_{m,\beta}
\longrightarrow N_{m,\beta}\otimes(\OO_{A_0}/\II_p^{m-1})\longrightarrow0
\]
and the zero-dimensional support of the last term, we obtain
\[
H^2(A_0,\mathcal K_{m,\beta})\simeq H^2(A_0,N_{m,\beta})=0.
\]
The last equality follows from K\"unneth, since both factors of $N_{m,\beta}$ have positive degree. The cohomology sequence of \eqref{eq:nonvanishingsequence} now gives a surjection
\[
H^1(A_0,\GG_m\otimes\pr_1^*\beta)\twoheadrightarrow
H^1(C_1,\OO_{C_1})\simeq\CC.
\]
Consequently
\[
\pr_1^*\Pic^0(E)\subset V^1(A_0,\GG_m).
\]
This is a codimension-one subvariety of $\Pic^0(A_0)$, so $\GG_m$ is not $M$-regular.

Suppose that $\GG_m$ admitted a Chen--Jiang decomposition. Since it is torsion-free of rank one, the decomposition would have a single nonzero summand:
\[
\GG_m\simeq\gamma\otimes q^*\mathcal M,
\]
where $q:A_0\to B$ has connected fibers, $\gamma$ is torsion, and $\mathcal M$ is $M$-regular. Flatness of $q$ shows that $\mathcal M$ is torsion-free of rank one. If $\dim B=2$, then $q$ is an isomorphism, so $\GG_m$ would be $M$-regular. If $\dim B\le1$, then $\mathcal M$ is locally free and hence so is $q^*\mathcal M$. This is also impossible: at $p$, the ideal $\II_p^m$ is $(x,y)^m$, which has $m+1$ generators and is not locally free of rank one. Thus $\GG_m$ has no Chen--Jiang decomposition.
\end{proof}

\begin{remark}\label{rem:plurisubsheaf}
For the pair and morphism of Theorem \ref{thm:pluri}, put
\[
\FF_m=\mu_*\OO_X\bigl(m(K_X+\Delta)-\Delta\bigr).
\]
Since $\Delta=\mu^*(C_1+C_2)-2R$,
\[
m(K_X+\Delta)-\Delta\sim\mu^*(m-1)(C_1+C_2)-(m-2)R,
\]
and therefore
\[
\FF_m\simeq L^{m-1}\otimes\II_p^{m-2},\qquad \II_p^0=\OO_{A_0}.
\]
This is the subsheaf of Proposition \ref{prop:subsheaf}. It satisfies $\mathrm{IT}^0$.

To verify the latter assertion, for integers $a,b>k\ge0$ set
\[
\mathcal H_{a,b,k}=\bigl(\pr_1^*P^a\otimes\pr_2^*P^b\bigr)\otimes\II_p^k.
\]
For $k\ge1$, the restriction sequence \eqref{eq:residual} gives
\[
0\longrightarrow\mathcal H_{a-1,b,k-1}\longrightarrow\mathcal H_{a,b,k}
\longrightarrow i_*P^{b-k}\longrightarrow0,
\]
where $C_1$ is identified with $E$. The sheaves $\mathcal H_{a,b,0}$ satisfy $\mathrm{IT}^0$ by K\"unneth. Since $b-k>0$, the same holds for $i_*P^{b-k}$. After tensoring the sequence by an arbitrary $\alpha\in\Pic^0(A_0)$, induction on $k$ gives $\mathrm{IT}^0$ for $\mathcal H_{a,b,k}$. Taking $a=b=m-1$ and $k=m-2$ proves the assertion for $\FF_m$.
\end{remark}

\section{Higher direct images and higher-dimensional bases}
Let $h:X\to E$ denote the morphism $f$ constructed in Theorem \ref{thm:lc}, and put $L=\OO_X(K_X+\Delta)$. Thus
\[
X=\Bl_{(0,0)}(E\times E),\qquad h=\pr_2\circ\mu,\qquad h_*L\simeq F_2.
\]
We use products with abelian varieties and the following vanishing.

\begin{lemma}\label{lem:vanishing}
Let $h:X\to E$ be the morphism of Theorem \ref{thm:lc}, where $X=\Bl_p(E\times E)$, $p=(0,0)$, and $\Delta$ is the sum of the strict transforms of $\{0\}\times E$ and $\Delta_E$. Then
\[
h_*\OO_X(K_X+\Delta)\simeq F_2,\qquad R^qh_*\OO_X(K_X+\Delta)=0\quad(q>0).
\]
\end{lemma}
\begin{proof}
Write $Y=E\times E$, $g=\pr_2$, and $h=g\circ\mu$. Put
\[
L_Y=\OO_Y(C_0+C_\Delta),\qquad M=\OO_X(K_X+\Delta).
\]
By the computation in Theorem \ref{thm:lc},
\[
M\simeq\mu^*L_Y\otimes\OO_X(-R).
\]
We first recall the elementary direct-image calculation for the blow-up $\mu:X\to Y$ at the smooth point $p$. The exceptional curve is $R\simeq\PP^1$, and there is an exact sequence
\[
0\longrightarrow\OO_X(-R)\longrightarrow\OO_X\longrightarrow\OO_R\longrightarrow0.
\]
For a blow-up of a smooth point one has
\[
\mu_*\OO_X=\OO_Y,\qquad R^1\mu_*\OO_X=0.
\]
Also $\mu_*\OO_R\simeq k(p)$ and $R^1\mu_*\OO_R=0$, because $R\simeq\PP^1$ and $H^1(\PP^1,\OO_{\PP^1})=0$. Pushing forward the exact sequence therefore gives
\[
0\longrightarrow\mu_*\OO_X(-R)\longrightarrow\OO_Y\longrightarrow k(p)
\longrightarrow R^1\mu_*\OO_X(-R)\longrightarrow0.
\]
The map $\OO_Y\to k(p)$ is evaluation at $p$, hence it is surjective and its kernel is $\II_p$. Thus
\[
\mu_*\OO_X(-R)=\II_p,\qquad R^1\mu_*\OO_X(-R)=0.
\]
There are no higher direct images for $q\ge2$, since the fibers of $\mu$ have dimension at most one. By the projection formula,
\[
R\mu_*M\simeq L_Y\otimes\II_p
\]
concentrated in degree zero. Consequently
\[
R^qh_*M\simeq R^qg_*(L_Y\otimes\II_p)\qquad(q\ge0).
\]
Now use the exact sequence already obtained in \eqref{eq:restriction}:
\[
0\longrightarrow\OO_Y(C_0)\longrightarrow L_Y\otimes\II_p\longrightarrow i_*\OO_{C_\Delta}\longrightarrow0.
\]
By \eqref{eq:product},
\[
R^qg_*\OO_Y(C_0)=0\qquad(q>0).
\]
Moreover, $g\circ i:C_\Delta\to E$ is an isomorphism, so
\[
R^qg_*i_*\OO_{C_\Delta}=0\qquad(q>0).
\]
The long exact sequence of higher direct images now gives
\[
R^qg_*(L_Y\otimes\II_p)=0\qquad(q>0).
\]
Hence $R^qh_*M=0$ for all $q>0$. The equality $h_*M\simeq F_2$ is exactly Theorem \ref{thm:lc}.
\end{proof}

\begin{theorem}[Higher direct images over an elliptic curve]\label{thm:higher}
Let $E$ be an elliptic curve and let $r\ge0$ be an integer. There exist a smooth projective variety $X_r$ of dimension $r+2$, a reduced SNC divisor $\Delta_r$ on $X_r$, and a morphism $f_r:X_r\to E$ such that
\[
R^jf_{r*}\OO_{X_r}(K_{X_r}+\Delta_r)\simeq F_2^{\oplus\binom rj}\qquad(0\le j\le r).
\]
None of these sheaves admits a Chen--Jiang decomposition. The higher direct images vanish for $j>r$.
\end{theorem}
\begin{proof}
Choose an abelian variety $A$ of dimension $r$ and set
\[
X_r=X\times A,\qquad\Delta_r=\pr_X^*\Delta,\qquad f_r=h\circ\pr_X.
\]
The pair is log smooth, and $K_A\sim0$ gives
\[
\OO_{X_r}(K_{X_r}+\Delta_r)\simeq\pr_X^*L,\qquad
R^q\pr_{X*}\pr_X^*L\simeq L\otimes H^q(A,\OO_A).
\]
The Leray spectral sequence has terms
\[
E_2^{p,q}=R^ph_*L\otimes H^q(A,\OO_A)
\ \Longrightarrow\ R^{p+q}f_{r*}\OO_{X_r}(K_{X_r}+\Delta_r).
\]
Lemma \ref{lem:vanishing} gives $E_2^{p,q}=0$ for $p>0$. Therefore
\[
R^jf_{r*}\OO_{X_r}(K_{X_r}+\Delta_r)
\simeq F_2\otimes H^j(A,\OO_A)
\simeq F_2\otimes\bigwedge^jH^1(A,\OO_A)
\simeq F_2^{\oplus\binom rj}.
\]
In particular, these sheaves do not admit a Chen--Jiang decomposition. Indeed, let $N>0$ and set $G=F_2^{\oplus N}$. Suppose that $G$ admitted such a decomposition. Since $\deg G=0$, the same degree argument as in Theorem \ref{thm:lc} shows that every summand must come from a point; thus $G$ would be a direct sum of torsion line bundles on $E$. On the other hand, Remark \ref{rem:F2loci} gives
\[
V^0(E,G)=\{\OO_E\}.
\]
For a direct sum of torsion line bundles $\bigoplus_i\alpha_i$, the set $V^0$ is precisely the set of points $\{\alpha_i^{-1}\}$. Hence every $\alpha_i$ must be trivial. Since $\rk G=2N$, this would give $G\simeq\OO_E^{\oplus2N}$ and therefore $h^0(E,G)=2N$. But $h^0(E,F_2)=1$, so $h^0(E,G)=N$, a contradiction. Thus no such decomposition exists. The vanishing for $j>r$ follows from $H^j(A,\OO_A)=0$ for $j>r$.
\end{proof}

\begin{corollary}
For every integer $n\ge2$ and every $0\le j\le n-2$, there exist a smooth projective variety $X_n$ of dimension $n$, a reduced SNC divisor $\Delta_n$, and a morphism $f_n:X_n\to E$ for which
\[
R^jf_{n*}\OO_{X_n}(K_{X_n}+\Delta_n)
\]
does not admit a Chen--Jiang decomposition.
\end{corollary}
\begin{proof}
Apply Theorem \ref{thm:higher} with $r=n-2$.
\end{proof}

\begin{theorem}[Higher-dimensional bases]\label{thm:base}
Let $E$ be an elliptic curve and let $A$ be an abelian variety of dimension $a$. For every integer $b\ge0$, there exist a smooth projective variety $W$ of dimension $a+b+2$, a reduced SNC divisor $\Delta_W$ on $W$, and a morphism $\varphi:W\to E\times A$ such that
\[
R^q\varphi_*\OO_W(K_W+\Delta_W)\simeq(p^*F_2)^{\oplus\binom bq}\qquad(0\le q\le b),
\]
where $p:E\times A\to E$ is the projection. None of these sheaves admits a Chen--Jiang decomposition, and the higher direct images vanish for $q>b$.
\end{theorem}
\begin{proof}
Choose an abelian variety $B$ of dimension $b$. Put
\[
W=X\times A\times B,\qquad\Delta_W=\pr_X^*\Delta,
\qquad\varphi(x,u,v)=(h(x),u).
\]
Let $\rho:W\to X\times A$ and $s:X\times A\to X$ be the projections, and set $h_A=h\times\mathrm{id}_A$. Then $\varphi=h_A\circ\rho$ and
\[
\OO_W(K_W+\Delta_W)\simeq\rho^*s^*L.
\]
Flat base change in the Cartesian square defined by $h$ and $p$ gives
\[
R^ih_{A*}s^*L\simeq p^*R^ih_*L=
\begin{cases}p^*F_2,&i=0,\\0,&i>0.\end{cases}
\]
Moreover,
\[
R^q\rho_*\rho^*s^*L\simeq s^*L\otimes H^q(B,\OO_B).
\]
Leray now yields
\[
R^q\varphi_*\OO_W(K_W+\Delta_W)
\simeq p^*F_2\otimes H^q(B,\OO_B)
\simeq(p^*F_2)^{\oplus\binom bq}.
\]
In particular, these sheaves do not admit a Chen--Jiang decomposition. To see this, set $D=E\times A$ and $G=(p^*F_2)^{\oplus N}$ for $N>0$. For $(\alpha,\beta)\in\Pic^0(E)\times\Pic^0(A)$, K\"unneth gives
\[
H^0\bigl(D,G\otimes(\alpha\boxtimes\beta)\bigr)
\simeq\bigl(H^0(E,F_2\otimes\alpha)\otimes H^0(A,\beta)\bigr)^{\oplus N}.
\]
Thus
\[
V^0(D,G)=\{\OO_D\},\qquad h^0(D,G)=N,\qquad\rk G=2N.
\]
If
\[
G\simeq\bigoplus_i\alpha_i\otimes p_i^*G_i
\]
were a Chen--Jiang decomposition, then by \cite[Lemma 3.3]{LPS}
\[
\{\OO_D\}=V^0(D,G)=\bigcup_i\alpha_i^{-1}\otimes p_i^*\Pic^0(D_i).
\]
Each $D_i$ must be a point and each $\alpha_i$ trivial. Hence $G\simeq\OO_D^{\oplus2N}$, contradicting $h^0(D,G)=N$. The vanishing for $q>b$ follows from $H^q(B,\OO_B)=0$ for $q>b$.
\end{proof}

\begin{corollary}\label{cor:all}
For every $d\ge1$ and $j\ge0$, there exist a projective log-smooth pair $(W,\Delta_W)$ with reduced boundary and $\dim W=d+j+1$, an abelian variety $B_d$ of dimension $d$, and a morphism $\varphi:W\to B_d$ such that
\[
R^j\varphi_*\OO_W(K_W+\Delta_W)\simeq p^*F_2
\]
for a quotient $p:B_d\to E$. In particular, this sheaf does not admit a Chen--Jiang decomposition.
\end{corollary}
\begin{proof}
Take $\dim A=d-1$ and $\dim B=j$ in Theorem \ref{thm:base}. Then
\[
B_d=E\times A,\qquad\dim W=2+(d-1)+j=d+j+1,\qquad h^j(B,\OO_B)=1.
\]
The assertion follows from the formula in that theorem.
\end{proof}

\begin{theorem}[Higher-dimensional examples for $m\ge2$]\label{thm:pluriproduct}
Let $\mu:X'\to A_0=E\times E$ and $(X',\Delta')$ be the construction of Theorem \ref{thm:pluri}. For every $d\ge2$ and $r\ge0$, choose abelian varieties $C$ and $T$ with $\dim C=d-2$ and $\dim T=r$, and set
\[
B_d=A_0\times C,\qquad W=X'\times C\times T,\qquad\Delta_W=\pr_{X'}^*\Delta'.
\]
Let $p:B_d\to A_0$ be the projection and define
\[
\varphi:W\longrightarrow B_d,\qquad\varphi(x,c,t)=(\mu(x),c).
\]
Then $(W,\Delta_W)$ is log smooth with reduced boundary, $\dim W=d+r$, and $\varphi$ is surjective with connected fibers. Furthermore, for every integer $m\ge2$,
\[
R^j\varphi_*\OO_W\bigl(m(K_W+\Delta_W)\bigr)
\simeq(p^*\GG_m)^{\oplus\binom rj}\qquad(0\le j\le r).
\]
None of these sheaves admits a Chen--Jiang decomposition, and the higher direct images vanish for $j>r$.
\end{theorem}
\begin{proof}
Put $M_m=\OO_{X'}(m(K_{X'}+\Delta'))$. We first show that
\begin{equation}\label{eq:blowupall}
\mu_*M_m\simeq\GG_m,\qquad R^q\mu_*M_m=0\quad(q>0).
\end{equation}
The first equality is Theorem \ref{thm:pluri}. For the vanishing, let $R$ be the exceptional curve. Since $M_m\simeq\mu^*L^m\otimes\OO_{X'}(-mR)$, it is enough to prove $R^1\mu_*\OO_{X'}(-kR)=0$ for $k\ge0$. The case $k=0$ is the usual vanishing for the blow-up of a smooth point. For $k\ge0$, push forward
\[
0\longrightarrow\OO_{X'}(-(k+1)R)\longrightarrow\OO_{X'}(-kR)
\longrightarrow\OO_R(k)\longrightarrow0.
\]
The map
\[
\II_p^k\longrightarrow\mu_*\OO_R(k)
\simeq\II_p^k/\II_p^{k+1}
\]
is surjective. The long exact sequence therefore proves the vanishing by induction. Degrees $q\ge2$ vanish because the fibers of $\mu$ have dimension at most one. This proves \eqref{eq:blowupall}.

Factor $\varphi$ as
\[
W=X'\times C\times T\xrightarrow{\ \rho\ }X'\times C
\xrightarrow{\ \nu=\mu\times\mathrm{id}_C\ }A_0\times C,
\]
and let $s:X'\times C\to X'$ be the projection. Since $K_C\sim K_T\sim0$,
\[
\OO_W\bigl(m(K_W+\Delta_W)\bigr)\simeq\rho^*s^*M_m.
\]
The product formula gives
\[
R^j\rho_*(\rho^*s^*M_m)\simeq s^*M_m\otimes H^j(T,\OO_T).
\]
By flat base change and \eqref{eq:blowupall},
\[
\nu_*s^*M_m\simeq p^*\GG_m,\qquad R^q\nu_*s^*M_m=0\quad(q>0).
\]
Thus the Leray spectral sequence gives
\[
R^j\varphi_*\OO_W\bigl(m(K_W+\Delta_W)\bigr)
\simeq p^*\GG_m\otimes H^j(T,\OO_T).
\]
Since $h^j(T,\OO_T)=\binom rj$ for $0\le j\le r$ and is zero for $j>r$, this is the asserted formula.

We next show that $H:=p^*\GG_m$ has no Chen--Jiang decomposition. For every $\alpha\in\Pic^0(A_0)$, the exact sequence
\[
0\longrightarrow\GG_m\otimes\alpha\longrightarrow L^m\otimes\alpha
\longrightarrow(L^m\otimes\alpha)\otimes(\OO_{A_0}/\II_p^m)\longrightarrow0
\]
gives $H^2(A_0,\GG_m\otimes\alpha)=0$ and
\[
\chi(A_0,\GG_m\otimes\alpha)
=m^2-\frac{m(m+1)}2=\frac{m(m-1)}2>0.
\]
Here $\length(\OO_{A_0,p}/\II_p^m)=m(m+1)/2$, by counting monomials of total degree less than $m$. It follows that $h^0(A_0,\GG_m\otimes\alpha)>0$ for every $\alpha$. K\"unneth therefore gives
\[
V^0(B_d,H)=\Pic^0(A_0)\times\{\OO_C\}=p^*\Pic^0(A_0).
\]
Since $H$ is torsion-free of rank one, a Chen--Jiang decomposition would have the form
\[
H\simeq\tau\otimes q^*\mathcal M,
\]
where $q:B_d\to Q$ has connected fibers, $\tau$ is torsion, and $\mathcal M$ is $M$-regular. By \cite[Lemma 3.3]{LPS},
\[
p^*\Pic^0(A_0)=\tau^{-1}\otimes q^*\Pic^0(Q).
\]
Both pullback images are abelian subvarieties, and the left side contains the origin. Hence
\[
q^*\Pic^0(Q)=p^*\Pic^0(A_0),\qquad \tau\in p^*\Pic^0(A_0).
\]
The quotients $p$ and $q$ consequently have the same kernel. Identifying $Q$ with $A_0$, we may write $q=p$ and $\tau=p^*\tau_0$. Applying $p_*$ to
\[
p^*\GG_m\simeq p^*(\tau_0\otimes\mathcal M)
\]
gives $\GG_m\simeq\tau_0\otimes\mathcal M$, because $p_*\OO_{B_d}=\OO_{A_0}$. This contradicts the failure of $M$-regularity proved in Theorem \ref{thm:pluri}. Finally, if a nonzero direct sum of copies of $H$ admitted a Chen--Jiang decomposition, then so would its direct summand $H$, by \cite[Proposition 3.6]{LPS}.
\end{proof}

\bigskip
\noindent\textsc{Department of Mathematical and Statistical Sciences, University of Alberta,\\
University Commons 5-140, Edmonton, Alberta, Canada T6G 2N8.}\\
\textit{Email:} \href{mailto:hbenamma@ualberta.ca}{hbenamma@ualberta.ca}
\end{document}